\documentclass[12pt,a4paper,leqn, reqno]{amsart}
\usepackage{color}
\usepackage{latexsym}
\usepackage{mathptmx}
\usepackage{amscd,amssymb,amsthm}
\usepackage{graphicx}
\def\ds{\displaystyle}
\numberwithin{equation}{section}
\theoremstyle{plain}
\newtheorem{theorem}{Theorem}[section]
\newtheorem{corollary}[theorem]{Corollary}
\newtheorem{lemma}[theorem]{Lemma}
\newtheorem{proposition}[theorem]{Proposition}

\theoremstyle{definition}
\newtheorem{definition}{Definition}[section]

\theoremstyle{remark}
\newtheorem{remark}{Remark}[section]
\newtheorem{remarks}{Remarks}
\newtheorem{example}{Example}[section]

\numberwithin{equation}{section}

\numberwithin{table}{section}

\numberwithin{figure}{section}

\def\ds{\displaystyle}

\newcommand{\R}{\mathbb{R}}
\newcommand{\N}{\mathbb{N}}
\newcommand{\Z}{\mathbb{Z}}

\newcommand{\bea}{\begin{eqnarray*}}
\newcommand{\eea}{\end{eqnarray*}}
\newcommand{\bean}{\begin{eqnarray}}
\newcommand{\eean}{\end{eqnarray}}

\newcommand{\K}{\mathbb{K}}
\newcommand{\NB}{\mathbf{N}}
\newcommand{\KC}{\widehat{\mathbb{K}}}

\newcommand{\al}{\alpha}

\newcommand{\la}{\lambda}

\newcommand{\xt}{{(x,t)}}

\newcommand{\lam}{{(\la,m)}}

\newcommand{\ND}{\mathbf{\mathcal{N}}}

\begin{document}

\title[Fourier multipliers for Hardy spaces on Laguerre hypergroup]
{\bf Fourier multipliers for Hardy spaces on Laguerre hypergroup}
\author{\bf Atef Rahmouni}
\address{University of Carthage\\
Faculty of Sciences of Bizerte\\
Department of Mathematics Bizerte  7021 Tunisia.}
\email{Atef.Rahmouni@fsb.rnu.tn}
\begin{abstract}
The main purpose of this paper is to give an estimate for the Fourier--Laguerre transform on Hardy spaces in the setting of Laguerre hypergroup. The atomic and molecular characterization is
investigated which allows us to prove a version of H\"ormander's multiplier theorem on $H^p~(0<p\leq1).$
\end{abstract}
\maketitle
\section{Introduction}

The theory of Fourier multipliers is well developed on Euclidean spaces, with various results having been established to give sufficient conditions for a multiplier operator to be bounded on the Lebesgue spaces $L^p (p > 1)$ or Hardy spaces $H^p (0 < p\leq 1).$ Among these are H\"{o}rmander's multiplier theorem and its variants.

A bounded measurable function $M$ defined on $\mathbb{R}^n$ is said to be a multiplier for $H^p(\mathbb{R}^n),\linebreak 0 < p \leq \infty,$ if
$f \in L^2\cap H^p(\mathbb{R}^n)$ implies $\mathcal{F}^{-1}(M\widehat{f})\in H^p(\mathbb{R}^n)$ and $$ \|\mathcal{F}^{-1}(M\widehat{f})\|_{H^p(\mathbb{R}^n)}\leq C_p \|f\|_{H^p(\mathbb{R}^n)} \qquad\mbox{(with $C_p$ independent of $f$)},$$ where $\widehat{f}$ and  $\mathcal{F}^{-1}(f)$ denote the Fourier transform and the inverse Fourier transform, respectively. The multiplier theorem was originally due to H\"{o}rmander \cite{Hormander} on $\mathbb{R}^n.$
Calder\'on and Torchinsky \cite{CT} extended the $L^p(\mathbb{R}^n)$ multipliers to $H^p(\mathbb{R}^n)$ multipliers. A considerable effort has been made to extend the classical Fourier multiplier theory to groups and Lie group.
De Michele and Mauceri \cite{DeM} applied Coifman
and Weiss'theory \cite[Chapter 3]{COW1} to extend the $L^p$ multiplier theorem to the three--dimensional Heisenberg group  $\mathbb{H}.$ Lin \cite{Lin,Lin1} followed the same general approach of \cite{DeM} and extended their result to the more general case of the $(2n+1)$--dimensional Heisenberg group $\mathbb{H}^n.$
In the setting of hypergroups a version of H\"ormander's multiplier theorem was established in \cite{ST11} for $L^p$ functions $(p>1)$ on Bessel--Kingman hypergroups, a particular class of Ch\'ebli--Trim\`eche hypergroups with polynomial growth. For general Ch\'ebli--Trim\`eche hypergroups the $L^p$--Fourier multipliers were investigated by Bloom and Xu \cite{BL1} and extended their result to $H^p$ when $0 < p\leq 1$ (cf.\cite{BL11}).\\

The purpose of this paper is to investigate on local Hardy spaces on Laguerre hypergroup. To establish a version of H\"ormander's multiplier theorem, namely H\"ormander--type multipliers on $H^p(\mathbb{K})$  for $0 <p\leq1,$ coupled with the atomic and molecular characterizations of the local Hardy spaces on Laguerre hypergroups, but  a  well--known problem is the characterization of  $\mathcal{F}(f)$ for $f\in H^p(\mathbb{K}).$ Elements of  $H^p(\mathbb{K}),$ where $p<1,$ are not $L^1(\mathbb{K})$ functions, so how is one supposed to define their Fourier--Laguerre transforms?
To do this we employ the method used by Taibleson and Weiss \cite{TW} on euclidean space also motivated by the treatment  of  \cite{Bownik,Liu,Xu,BL1} which was the subject of the following result :


For $p\in(0, 1],$ the Fourier--Laguerre transform $\mathcal{F}(f)$ of $f\in H^p(\mathbb{K})$ is continuous function and satisfies the following estimate (cf. Theorem \ref{theorem6.6})
\begin{equation}\label{(4.100)}
|\mathcal{F}(f)(\lambda,m)|\leq C \|f\|_{H^p(\mathbb{K})}
\mathcal{N}(\lambda,m)^{\frac{Q}{2}(\frac{1}{p}-1)},
\end{equation}
where $\mathcal{N}(\lambda,m)$ the quasi--semi--norm on the dual of Laguerre hypergroup and $Q$ the homogenous dimension of $\mathbb{K}$ (cf. $\S 2$).

At the origin, the estimate (\ref{(4.100)}) forces
$f\in H^p\cap L^{1}(\mathbb{K})$ to have vanishing moments, as seen by the degree of 0 of $\mathcal{F}(f)$ at the origin, illustrating the necessity of the vanishing moments of the atoms. Away from the origin,
we show in Corollary \ref{t3.3.11} that the function
$|\widehat{f}(\lambda,m)|^{p} |\mathcal{N}(\lambda,m)|^{\frac{Q}{2}(p-2)}$
is integrable for $0<p\leq1,$ which is a generalization of Paley--type inequality on Laguerre hypergroup and by using the Marcinkiewicz interpolation theorem we obtain an analogue of the Paley--type inequality was extended (the range of $p$) to $L^{p}_{\alpha}(\mathbb{K}),~1<p\leq2,$
that is a Pitt--type inequality for the Fourier--Laguerre transform.\\

In the main result [Theorem \ref{multiplier}], when paired with the molecular characterization of local Hardy space on Laguerre hypergroup,
we showed that the multiplier operator $T_M(f)=\mathcal{F}^{-1}(M\widehat{f})$ is bounded provided the multiplier $M$ satisfies Mihlin-H\"ormander condition. The strategy of proof follows the classical approach, starting with kernel estimates. It will be more complicated in the general case, when the molecules are defined away from the origin.\\

The paper is organized as follows. We recall some results about harmonic analysis on $\mathbb{K}$ which we will need in the sequel are given in
$\S.$ 2. In $\S.$ 3 we give an appropriate definition of atoms  and molecules and investigate the molecular characterization of Hardy spaces $H^p(\mathbb{K})$ for $0<p\leq 1.$ $\S.$ 4, contains the main results of this paper, where some useful estimates for characters, prove of Theorem \ref{theorem6.6} and as we list some consequences of this theorem. Finally use the molecular characterization to obtain a version of H\"ormander's multiplier theorem and give application of H\"ormander's multiplier.
Finally, we mention that $C$ will be always used to denote a
suitable positive constant that is not necessarily the same in
each occurrence.

\section{Preliminaries on Laguerre hypergroup}

Most results in this section are known. We consistently use the same notations as those in \cite{BL, Trimeche}, and refer readers to \cite{S3} and \cite{BL, Trimeche} for most terminologies, notations, and detailed proofs.

Throughout the paper we denote by $\mathbb{K} = [0,+\infty) \times \mathbb{R}$
the Laguerre hypergroup which is the fundamental manifold of the radial function space for the Heisenberg group \cite{BL}. We recall that $(\mathbb{K}, \ast_{\alpha})$ is a commutative hypergroup  \cite{NT},
on which the involution and the Haar measure are respectively given by the
homeomorphism $(x, t) \rightarrow (x, t)^{-} = (x,-t)$ and the
Radon positive measure $dm_{\alpha} (x, t) = \frac {x^{2\alpha
+1}}{\pi \Gamma(\alpha+1)} dxdt.$  The unit element of
$(\mathbb{K}, \ast_{\alpha})$ is given by $ e = (0, 0),$ i.e. $\delta
_{(x,t)} \ast_{\alpha} \delta_{(0,0)} = \delta_{(0,0)} \ast_{\alpha} \delta_{(x,t)}
= \delta_{(x,t)} $ for all $(x, t) \in\mathbb{K}.$ The convolution product $ \ast_{\alpha}
$ is defined for two bounded Radon measures $ \mu $ and $ \nu $ on $
\mathbb{K}$ as follows
$$ \langle \mu  \ast_{\alpha} \nu, f \rangle = \int_{\mathbb{K} \times \mathbb{K}}
 T ^{(\alpha)} _{(x,t)} f (y, s)d\mu(x, t)d\nu(y, s),$$
where $\alpha $ is a fixed nonnegative real number and $\{T
^{(\alpha)} _{(x,t)}\}_{(x,t)\in \mathbb{K}} $   are the translation
operators on the Laguerre hypergroup (cf. \cite{BL, NT,
Stempak, Trimeche}), given by
\begin{displaymath}
T ^{(\alpha)} _{(x,t)} f (y, s)=\langle \delta_{(x,t)} \ast_{\alpha}
\delta_{(y,s)}, f \rangle = \left\{
\begin{array}{ll} \frac{\alpha}{\pi}\int_{0}^{1}\int_{0}^{2\pi}f((\xi, \eta)_{r,\theta})r(1-r^{2})^{\alpha-1}d\theta dr
&  \textrm{if $ \alpha>0 ,$}\\\\
\frac{1}{2\pi}\int_{0}^{2\pi}f((\xi, \eta)_{1,\theta}) d\theta &
\textrm{ if $ \alpha = 0 ,$}
\end{array} \right.
\end{displaymath}
where
$(\xi,\eta)_{r,\theta}=(\sqrt{x^{2}+y^{2}+2xyr\cos\theta},t+s+xyr\sin\theta).$
Note that $T ^{(\alpha)} _{(0,0)} f (y, s)=f(y, s).$

For the particular case $ \mu = fm_{\alpha}$ and $\nu =
gm_{\alpha}, $ $ f $ and $ g $ being two suitable functions on $
\mathbb{K},$ one has  $ \mu \ast_{\alpha}\nu = ( f \ast_{\alpha}
g)m_{\alpha},$ where $ f \ast_{\alpha} g  $ is the convolution
product of $ f $ and $ g $ given by  $$ f \ast_{\alpha} g(x, t) =
\int_{\mathbb{K} \times \mathbb{K}} T ^{(\alpha)} _{(-y, s)}
f (x, t)g(y,s)dm_{\alpha}(y, s).$$
The  harmonic analysis on the Laguerre hypergroup is generated by
the Laguerre operator
\begin{equation}\label{L}
\ds \mathfrak{L}_\al=\frac{\partial ^{2}}{\partial x^{2}}+\frac{2\alpha +1}{x}\frac{\partial }{\partial x}+x^{2}\frac{\partial ^{2}}{\partial t^{2}}, \quad \alpha \geq 0.
\end{equation}
For $\alpha= n-1, \, n\in
\mathbb{N}\setminus\{0\}, \,  \mathfrak{L}_{n-1}$ is the radial
part of the sub--Laplacian  of the Heisenberg group $\mathbb{H}^n.$
For all $(x,t)\in \mathbb{K} = [0, +\infty[\times \mathbb{R}$ and
$(\lambda,m)\in  \widehat{\mathbb{K}}
=\mathbb{R}\setminus\{0\}\times \mathbb{N}$ we denote by
$\varphi^{\alpha}_{(\lambda,m)}$ the function given by
$$\varphi^{\alpha}_{(\lambda,m)}(x,t)=e^{i\lambda
t}\mathcal{L}_{m}^{\alpha}(|\lambda|x^{2})$$ where
$\mathcal{L}_{m}^{\alpha}$ is the Laguerre function defined on
$[0, +\infty[ $ by
$\mathcal{L}_{m}^{\alpha}(x)=e^{-x/2}\frac{L^{\alpha}_{m}(x)}
{L^{\alpha}_{m}(0)},$
and $L^{\alpha}_{m}$ being the Laguerre polynomial of degree $m$
and order $\alpha $ \cite{MN, NT}.
$\varphi^{\alpha}_{(\lambda,m)}$ are eigenfunctions of the
operator $\mathfrak{L}_\al,$  namely one has
for any function $f\in L^{1}_{\alpha}(\mathbb{K}),$ with
$\mathfrak{L}_{\alpha} f\in
L^{1}_{\alpha}(\mathbb{K})$
\begin{equation}\label{spectre}
(\mathfrak{L}_\al f\widehat{)}\lam=-\ND\lam\widehat{f}\lam
\end{equation}
where $\ND(\la,m)=4|\la|(m+{\al+1\over2})$ is a quasi--norm on
$\widehat{\mathbb{K}} $ (cf. \cite{ST1}).\\

The dual  $\widehat{\mathbb{K}} $ (see \cite{BL}) of Laguerre hypergroup is the space of all bounded continuous and multiplicative
functions $\chi : \mathbb{K} \rightarrow \mathbb{C}$  such that
$\widetilde{\chi}= \chi,$ where
$\widetilde{\chi}(x,t)=\overline{\chi}(x,-t),~ (x,t)\in
\mathbb{\mathbb{K}}.$ Furthermore, $\widehat{\mathbb{K}} $ is the
collection $\{\varphi^{\alpha}_{(\lambda,m)}; (\lambda,m)\in
\mathbb{R}^{*} \times \mathbb{N}\}\cup \{\varphi^{\alpha}_{\rho};
\rho \geq 0\},$ where $\varphi^{\alpha}_{(\lambda,m)}$ are
Laguerre functions and $\varphi^{\alpha}_{\rho}(x) = j_{\alpha}(\rho
x)$ are  Bessel functions of first kind and order $ \alpha $
\cite{MN}.  The dual of the Laguerre hypergroup
$\widehat{\mathbb{K}}$ can be topologically identified with the
so--called Heisenberg fan \cite{FH}, i.e., the subset embedded in
$\mathbb{R}^{2}$ given by
$$\ds\cup_{m\in\N} \{ (\la,\mu)\in\R^2 : \mu= |\la|(2m+\al+1), \la\neq 0
\}\cup\{(0,\mu)\in\R^2 : \mu\geq0\}\sim
\{\mathbb{R}\setminus\{0\}\times\mathbb{N}\}\cup\{(0,0)\}.$$
Moreover, the subset $\{(0,\mu)\in\R^2 : \mu\geq0\}$ has zero
Plancherel measure, therefore it will be usually disregarded. The
topology on $\mathbb{K}$ is given by the norm $\mathbf{N}(x, t)=
(x^{4}+4t^{2})^{1/4},$ while we assign to $\widehat{\mathbb{K}}$
the topology generated by the quasi--semi--norm
$\mathcal{N}(\lambda,m)=4|\lambda|(m+\frac{\alpha+1}{2}).$\\
The Fourier--Laguerre transform of a suitable function $f$ on
$\mathbb{K}$ is given by
\begin{equation}\label{Fourier}
\widehat{f}(\lambda,m)=\int_{\mathbb{K}}
\varphi^{\alpha}_{(-\lambda,m)}(x,t) fdm_{\alpha}(x,t).
\end{equation}
It is well known that the Fourier--Laguerre transform given above is a
topological isomorphism from the Schwartz space on $\mathbb{K}$
onto $\mathcal{S}(\widehat{\mathbb{K}})$: the Schwartz space on
$\widehat{\mathbb{K}}$ (see \cite{AM, NT}). Its inverse is given
by
$$g^{\vee}=\int_{\widehat{\mathbb{K}}}\varphi^{\alpha}_{(\lambda,m)} g
d\gamma_{\alpha}(\lambda,m)$$ where $d\gamma_{\alpha} $  is the
Plancherel measure on $\widehat{\mathbb{K}}$ given by
$d\gamma_{\alpha}(\lambda,m)=|\lambda|^{\alpha+1}L^{\alpha}_{m}(0)
\delta_{m}\otimes
d\lambda.$

We denote by $ L^{p}_{\alpha}(\mathbb{K})$ (resp.
$L^{p}_{\alpha}(\mathbb{\widehat{K}})$ where $ 1 \leq p\leq \infty
$ the $p$--th Lebesgue space on $\mathbb{K}$ (resp. on $
\mathbb{\widehat{K}} $) formed by the measurable functions $ f :
\mathbb{K}\rightarrow \mathbb{C} $ (resp. $ \Phi :
\mathbb{\widehat{K}}\rightarrow \mathbb{C}$) such that $\|f\|
_{L^{p}_{\alpha} (\mathbb{{K}})} <+\infty $ (resp. $\|\Phi\|
_{L^{p}_{\alpha} (\mathbb{\widehat{K}})} <+\infty)$   where
\begin{displaymath}
\|f\|_{L^{p}_{\alpha}(\mathbb{K})} = \left\{ \begin{array}{ll}
\Bigg(\ds\int_{\mathbb{K}}|f(x,t)|^{p}dm_{\alpha}(x,t)\Bigg)^{1/p}&
\textrm{if $
p\in [1,+\infty[ ,$}\\
{\displaystyle{ess\:sup_{(x,t)\in \mathbb{K}}}}|f(x,t)|& \textrm{ if
$ p=+\infty ,$}
\end{array} \right.
\end{displaymath}
and
\begin{displaymath}
\|\Phi\|_{L^{p}_{\alpha}(\mathbb{\widehat{K}})} = \left\{
\begin{array}{ll}
\Bigg(\ds\int_{\mathbb{\widehat{K}}}|\Phi(\lambda,m)|^{p}
d\gamma_{\alpha}(\lambda,m)
\Bigg)^{1/p}&
\textrm{if $ p\in [1,+\infty[ ,$}\\
{\displaystyle{ess\:sup_{(\lambda,m)\in
\mathbb{\widehat{K}}}}}|\Phi(\lambda,m)|& \textrm{ if $ p=+\infty
.$}
\end{array} \right.
\end{displaymath}
We have the following Plancherel formula
$$ \|f\|_{L^{2}_{\alpha}({\mathbb{K}})} =
\|\widehat{f}\|_{L^{2}_{\alpha}(\widehat{\mathbb{K}})},\quad f\in
L^{1}_{\alpha}(\widehat{\mathbb{K}})\cap
L^{2}_{\alpha}(\widehat{\mathbb{K}}),$$ and we have
\begin{equation}\label{33}
\|\widehat{f}\|_{L^{\infty}_{\alpha}}\leq\|f\|_{L^{1}_{\alpha}}.\\
\end{equation}

For $\delta >0$ we define the dilation on $\mathbb{K}$ by
$\rho_\delta(x,t)=({x\over\delta},{t\over\delta^2}), \, (x,t)\in
\mathbb{K} .$   We will denoted by
$Q=2\al+4$ the homogenous dimension of $\mathbb{K}$ and by
$f_{\delta}(x,t) =\delta^{-Q}f\circ\rho_{\delta}(x,t)=
\delta^{-Q}f({x\over\delta},{t\over\delta^2})$ the dilated of the
function $f$ defined on $\mathbb{K}$ with respect to the measure $dm_\al$ in the sense that
$\|f_{\delta}\|_{L^{1}_{\alpha}(\mathbb{{K}})}
=\|f\|_{L^{1}_{\alpha}(\mathbb{{K}})}.$
We will define the  ball centered at $u=(x_0, t_0)$ of radius $r,$ i.e., the set
$$B(u,r)=\{(x, t)\in\mathbb{K}:~\mathbf{N}(x-x_0, t-t_0)<r \},$$
and we denote by $B(e,r) = \{(x, t)\in\mathbb{K}:~
\mathbf{N}(x, t)<r\} $ the open ball centered at $e$  with radius $r.$ The volume of the ball $B(u, r)$ is $C_Q~r^Q,$ where $C_Q$ is the volume of the unit ball $B(e,r).$
Let $I=(i^{1},i^{0})\in\mathbb{N}_{+}\times \mathbb{N}_{+},$ where $\mathbb{N}_{+}$ the set of all nonnegative integers, we set $d(I)=i^{1}+2i^{0}.$ If $P(x,t)=\sum_{I}a_{I}(x,t)^{I} $ is a polynomial where $(x, t)^I = x^{i^{1}}t^{i^{0}},$ then we call $\max\{d(I) : a_I \neq 0\}$ the homogeneous degree of $P(x, t).$ The set of all polynomials whose homogeneous degree $\leq s$ is denoted by $\mathcal{P}_s.$

\section{The Atomic Decomposition and Molecular Characterization}

The theory of Hardy spaces was also established on more general groups than $\mathbb{R}^n.$ Motivated by the work of Coifman and Weiss in \cite{COW,COW1} and others (have been introduced and studied Hardy spaces $H^p(X)$  on a space of homogeneous type) we introduced Hardy spaces on the Laguerre hypergroup which may be interpreted as Hardy spaces defined on a space of \textit{homogeneous type} $(X,\nu,\rho).$ By this we mean a topological space $X$ equipped with a continuous pseudometric $\rho$ and a positive measure $\nu$ satisfying
\begin{equation}\label{dublecondition}
\nu(E(\xi,2r)) \le  C \nu(E(\xi,r))
\end{equation}
with a constant $C$  independent of $\xi$  and $r>0$ and
$E(\xi,r)=\{\eta\in X \,:\, \rho(\xi,\eta)<r \},\
\rho(\xi,\eta)=|\xi-\eta|.$ We shall use this result in the case in which $X=\mathbb{K},$
$\xi=(x,t),\; \eta=(y,s) \in \mathbb{K},$
$\rho(\xi,\eta)=\max\{|x-y|,(t-s)^2 \},$
$d\nu(\xi) = d m_{\alpha}(x,t).$ It is clear that this measure satisfies the doubling condition (\ref{dublecondition}). The theory of Hardy spaces has been extended to other settings including homogeneous Lie groups, compact Lie groups and subsets of $\mathbb{R}^n$ etc. Some general references involving harmonic analysis and $H^p$ spaces are  \cite{CO,FS,S3,GR,St1}. It is well known that $H^p(X)$ can be defined either in terms of maximal functions or in terms of atomic decompositions (cf.\cite{S3}).
Namely, let $f\in \mathcal{S}'(\mathbb{K})$ the maximal function is defined by
$$\mathcal{M}(f)(x,t) = \sup_{\delta>0} |f\ast_{\alpha}\phi_\delta(x,t)|$$
where $\phi $ belongs to $\mathcal{S},$ the Schwartz space of
rapidly decreasing smooth functions  satisfying $\int \phi(x,t) d m_{\alpha}(x,t) =1.$ The delation $\phi_\delta$ is given by  $\phi_\delta(x, t) =\delta^{-Q}\phi(\frac{x}{\delta},\frac{t}{\delta^2}).$\\
We say that a tempered distribution $f\in\mathcal{S}'(\mathbb{K})$ is in $H^p(\mathbb{K})$ if $\mathcal{M}(f)$ is in $L^p_{\alpha}(\mathbb{K}).$ The
quasi--norm on $H^p$ is  $\|f\|^{p}_{H^p}=\|\mathcal{M}(f)\|^{p}_{L^p_{\alpha}},$ which
satisfies $\|f+g\|^{p}_{H^p}\leq \|f\|^{p}_{H^p}+\|g\|^{p}_{H^p}$ for
$0 < p \leq 1.$ When $p > 1, H^p$ and $L^p_{\alpha}$ are essentially the same since, by the celebrated theorem of Hardy and Littlewood,
$\|\mathcal{M}(f)\|^{p}_{L^p_{\alpha}}\leq C_{p}\|f\|^{p}_{L^p_{\alpha}};$ however, when $p\leq 1$ the space $H^p$ is much better adapted to problems arising in the theory of harmonic analysis.

Now, we describe the atomic decomposition, molecular characterization, and some properties for $H^p(\mathbb{K})$ centered at the origin $e\in\mathbb{K}$ which greatly simplifies the analysis of Hardy spaces which will be used later.
\begin{definition}
Let  $0 < p \leq 1\leq q \leq \infty$ with $p\neq q$ and $s\geq [ Q(1/p-1)].$ (Such an ordered triple $(p, q,s)$ is called admissible).
A function $a \in L^{q}_{\alpha}(\mathbb{K}),$ is called a $(p, q, s)$--atom with the center at the origin $e\in\mathbb{K},$ if it satisfies the following conditions
\begin{itemize}
\item[(i)] Supp $ a \subset B(e,r);$
\item[(ii)]$\|a\|_{L^q_\alpha(\mathbb{K})}\leq m_{\alpha}B(e,r)^{\frac{1}{q}-{1\over p}} =
    C~r^{Q(\frac{1}{q}-\frac{1}{p})};$
\item[(iii)]$\ds \int_{\mathbb{K}} a(x,t)P(x,t) dm_{\alpha}(x,t) = 0 $, ~~ for all $P\in \mathcal{P}_{s},$ with $ s\geq I = [Q(1/p
- 1)]$  where [ . ] denotes, as usual, the
``greatest integer not exceeding" function.
\end{itemize}
\end{definition}
Here, $(i)$ means that an atom must be a function with compact support, $(ii)$ is the size condition of atoms, and $(iii)$ is called the cancelation moment condition.
It is clear that $a\in L^p_{\alpha}(\mathbb{K})$ and $\|a\|_{L^p_{\alpha}(\mathbb{K})}\leq 1$ for any $(p, q, s)$--atom $a,$ by choosing $r = q/p > 1,
1/r' = 1- 1/r = 1- p/q,$ and using H\"older's inequality
\begin{eqnarray*}
\int_{\mathbb{K}} |a(x,t)|^p~dm_{\alpha}(x,t)&\leq& \Big(\int_{\mathbb{K}} (|a(x,t)|^p)^{r}~dm_{\alpha}(x,t)\Big)^{1/r}
\Big(\int_{B(e,r)}dm_{\alpha}(x,t)\Big)^{1/r'}\\
&=&\|a\|_{L^q_{\alpha}(\mathbb{K})}^{p}m_{\alpha}B(e,r)^{1-p/q}\leq 1.
\end{eqnarray*}
We say that $a$ is an exceptional atom when $\|a\|_{L^\infty_{\alpha}(\mathbb{K})}\leq 1.$
\begin{remark}
For every $(x_0,t_0),(x,t)\in\mathbb{K}$ and $r>0$ the function
$T^{(\alpha)}_{(x_0,t_0)} \chi_{B_{r}}(x,t)$ is supported in
$\widetilde{B}_r(x_0,t_0)$ and the following inequality is valid
(see \cite{Guliyev})
$$m_{\alpha}B_{r}(x_0,t_0)\leq Cr^Q\max\{1,(x_0/r)^{2\alpha+1}\} $$
$$m_{\alpha}\widetilde{B}_{r}(x_0,t_0)\leq Cr^Q\max\{1,(x_0/r)^{2\alpha+3}\} ,$$
where $\widetilde{B}_{r}(x_0,t_0)=\{(x,t)\in\mathbb{K} : |x-x_0|<r,
|t-t_0|<x_0(x_0+r)\}.$
So, $a(p, q, s)$--atom centered at $(x_0,t_0)\in\mathbb{K}$ is defined to be a $L^q_{\alpha}(\mathbb{K})$ function $a$ on $\mathbb{K}$ such that the translation $T^{(\alpha)}_{(x_0,t_0)} (a)(x,t)$ is not $a(p, q, s)$--atom centered at the origin.
Then the Hardy spaces associated with Laguerre hypergroup can be
regarded as the local version of $H^{p}.$
\end{remark}
\begin{theorem}(Atomic decomposition of $H^p$ \cite[Chapter 3]{S3})
Let $(p, q, s)$ be an admissible triple. Then any $f$ in $H^p$ can be represented as a linear combination of $(p, q, s)$--atoms
$f =\sum_{k=1}^{\infty}\beta_{k}a_{k},~\beta_{k}\in\mathbb{C},$ where the $a_{k}$ are $(p, q, s)$--atoms and the sum converges in $H^p.$ Moreover,
$\|f\|_{H^p} \thickapprox \inf \Big\{ \sum_{k=1}^\infty |\beta_{k}|^p :\sum_{k=1}^{\infty}\beta_{k}a_{k} \mbox{ is a decomposition of $f$
into\, }  atoms\Big\}.$
\end{theorem}
Let $0<p\leq 1.$ Our Hardy space $H^{p}_{q,s}(\mathbb{K})$
is constituted by all those $ f\in \mathcal{S}'(\mathbb{K})$  that can be represented by
\begin{equation}\label{sum}
f=\sum_{k=0}^{\infty} \beta_{k}a_{k}
\end{equation}
being $ \beta_{k}\in \mathbb{C}$ and $a_{k}$ is a
$p$--atom, for all $k\in\mathbb{N},$ where
$\sum_{k=0}^{\infty}|\beta_{k}|^{p}<\infty$ and the series in
(\ref{sum}) converges in $\mathcal{S}'(\mathbb{K}).$

We define on $H^{p}_{q,s}(\mathbb{K})$  the norm $\|-\|_{H^{p}_{q,s}(\mathbb{K})}$ by
$$ \|f\|_{H^{p}_{q,s}(\mathbb{K})}:=
\inf\Big(\sum_{k=0}^{\infty}|\beta_{k}|^{p}\Big)^{1/p},$$
where the infimum is taken over all those sequences
$\{\beta_{k}\}_{k\in\mathbb{N}}\subset \mathbb{C}$ such that $f$
is given by (\ref{sum}) for certain $p$--atoms $a_{k}, k\in\mathbb{N}.$
\begin{remarks}
\begin{verse}
\end{verse}
\textbullet~ While each $H^p_{q,s}(\mathbb{K})$ function has a decomposition into $(p, q, s)$--atoms, it is natural to compare the spaces of functions admitting decompositions into $(p, q, s)$--atoms and $(p, q', s')$--atoms. It was shown in \cite{S3} that for each $p,$ these spaces corresponding to different $(q, s)$ all coincide, hence we are dealing with a single space and so may be denoted simply by $H^{p}(\mathbb{K}).$
It is therefore convenient to let $q = 2,$ so the use of Plancherel's
formula becomes a powerful tool for the study of $H^p.$\\
\textbullet~\:  $H^p(\mathbb{K})$ is not a subspace of $L^p(\mathbb{K})$ when $p<1.$\\
Indeed, with the help of properties $(i),
(ii)$ and $(iii)$ for $a(p,\infty,s)$--atoms of $H^{p}(\mathbb{K}).$ We have
\begin{eqnarray}\label{Lp3.2} \|a\|_{L^{p}_{\alpha}(\mathbb{K})}=\Bigg(\int_{B(e,r)}|a(x,t)|^{p}dm_{\alpha}(x,t)
\Bigg)^{1/p}&\leq&  |B(e,r)|^{1/p}~\|a\|_{L^{\infty}_{\alpha}(\mathbb{K})}\nonumber \\
&\leq & C_{Q}^{1/p}~r^{\frac{Q}{p}}~r^{-\frac{Q}{p}}\nonumber\\
&=&C_{Q}^{1/p}.
\end{eqnarray}
Now, let $f\in H^{p}(\mathbb{K}).$ Then, $f\in\mathcal{S}'(\mathbb{K})$ and $f =\sum_{k}\beta_{k}a_{k},$
where $\beta_{k}\in\mathbb{C},$ and $a_{k}$ is $a(p,\infty,s)$--atom,
for every $k\in \mathbb{N},$ and $\sum_{k}|\beta_{k}|^{p}<\infty.$
By (\ref{Lp3.2} ), the series defining $f$ converges in $L^p_{\alpha}(\mathbb{K}).$ Hence, $f\in L^p_{\alpha}(\mathbb{K})$
and $\|f\|_{L^p_{\alpha}(\mathbb{K})}\leq C_{Q}^{1/p} \sum_{k}|\beta_{k}|.$
Using the fact that $\sum_{k}|\beta_{k}|\leq
\Big(\sum_{k}|\beta_{k}|^{p}\Big)^{1/p},$ we obtain
$$\|f\|_{L^p_{\alpha}(\mathbb{K})}\leq C_Q^{1/p} \|f\|_{H^p(\mathbb{K})}.$$
As a calculation shows, the atomic decomposition of an $f$ in $H^p$ does converge in $L^p,$ but in general it does not converge in $L^p$ to $f.$  For example, the difference $\delta(x,t) - \delta(b^{-1}(x,t))$ (where delta = Dirac delta--function and $b$ is not the identity element) lies in $H^p$ for $Q/(Q+1) < p < 1.$  It doesn't lie in $L^p$ because it's only a distribution, not a function; and any atomic decomposition of it will converge to 0 in the $L^p$ topology.  This point is raised in the book
of Folland--Stein \cite[Remark (2.18) (page 79)]{S3} .
\end{remarks}
Let us now introduce the molecules corresponding to the atoms we have just defined.
\begin{definition}
Let $0 < p \leq 1 \leq q \leq \infty,  p \neq q,$ and $ s\geq I = [Q(1/p
- 1)].$ Set $a=1-\frac{1}{p}+\varepsilon,~b=1-\frac{1}{q}+\varepsilon,$
where $\varepsilon >\max\{\frac{s}{Q},\frac{1}{p}-1\}.$
A function $M\in L^{q}_{\alpha}(\mathbb{K})$ called $a(p,q,s,\varepsilon)$--molecule centered at the origin $e\in\mathbb{K}.$ If it satisfies
\begin{itemize}
\item[(i)] $M(.,.).T^{(\alpha)}_{(0,0)}\mathbf{N}(.,.)^{Qb}\in L^q_\alpha(\mathbb{K});$
\item[(ii)] $\mathfrak{N}(M)=\|M\|^{\frac{a}{b}}_{L^q_\alpha(\mathbb{K})}
\|M(.,.).T^{(\alpha)}_{(0,0)}\mathbf{N}(.,.)^{Qb}\|^
{1-\frac{a}{b}}_{L^q_\alpha(\mathbb{K})}
<\infty ~~\\\mathfrak{N}(M)~~\mbox{is called the molecular norm of} ~~$M$;$
\item[(iii)] $\ds \int_{\mathbb{K}} a(x,t)P(x,t) dm_{\alpha}(x,t) = 0 $, \quad for all $P\in \mathcal{P}_{s}.$
    \end{itemize}
\end{definition}
The following result is very useful in establishing boundedness of linear operators
on $H^p.$
\begin{theorem}\cite{COW,TW}\label{theorem3.3}
\begin{itemize}
\item[(i)] Every $(p, q, s')$--atom $f$ is $a(p, q, s, \varepsilon)$--molecule for any
$\varepsilon>\max\{s/Q, 1/p-1\}, s \leq s'$ and $\mathfrak{N}(f)\leq C_1,$ where the constant $C_1$ is independent of the atom.
\item[(ii)] Every $(p, q, s, \varepsilon)$--molecule $M$ is in $H^p$ and $\|M\|_{H^p}\leq C_{2}\mathfrak{N}(M),$ where the constant
$C_2$ is independent of the molecule.
\end{itemize}
\end{theorem}
We can define $H^{p}(\mathbb{K})$ as follows (see Taibleson and Weiss \cite{TW}). Let $0<p\leq 1,~ f\in H^{p}(\mathbb{K})$ if and only if
$$f=\sum_{k}\beta_{k}M_{k}~~\mbox{in the sense of distribution,}$$
where $M_{k}$ are $a(p,q,s)$--molecular with $\mathfrak{N}(M_{k})<C $
for all $k$ and $\sum_{k}|\beta_{k}|^{p}<\infty.$
Moreover,
$$\|f\|_{H^p(\mathbb{K})} \thickapprox \inf \Big\{ \sum_{k=1}^\infty |\beta_{k}|^p :
\sum_{k=1}^{\infty}\beta_{k}M_{k} \mbox{ is a molecular decomposition of} ~~f\Big\}.$$

\begin{remark}
As a consequence of the atomic and molecular characterizations of
$H^p,$ to show that a linear map $T$ is bounded on $H^p$ it suffices to show that whenever $f$ is a $p$--atom then $Tf$ is a $p$--molecule and
$\mathfrak{N}(Tf)\leq C$ for some constant $C$ independent of $f.$
\end{remark}

\section{Main Result and Applications}

\subsection{Estimates for characters}
First we stat  the following lemma which has its own interest.
\begin{lemma}\label{lemma2}
For all $\lam\in \KC$ and $\xt\in \K$ and for any integer $I\in \N$
large enough, the character $\varphi^\al_\lam\xt$ satisfies
\begin{equation}
\label{estimate2}
\Big{|}{\partial^I\over\partial\la^I}\varphi^\al_\lam\xt\Big{|}\leq
C~\NB\xt^{2I},
\end{equation}
and
\begin{equation}
\label{estimate3}
\Big{|}{\partial^I\over\partial\la^I}\varphi^\al_\lam\xt\Big{|}\leq
C~\NB\xt^{2I}\varphi^\al_\lam\xt.
\end{equation}
\end{lemma}
\begin{proof}
\textbullet~ First, let us prove (\ref{estimate2}). By using Leibnitz's formula for the I'th derivative of a
product and the recurrence identity for Laguerre polynomials
${\partial^I\over\partial\la^I}L_m^\al(z)=(-1)^IL_{m-I}^{\al+I}$
we obtain for $\la>0$
\begin{eqnarray}\label{estimate4}
{\partial^I\over\partial\la^I}\varphi^\al_\lam\xt &=&
\sum_{j=0}^I\binom{I}{j}(-it)^{I-j}e^{-i\la t}{\partial^j\over\partial\la^j}[\mathcal{L}_m^\al(|\la|x^2)]\nonumber\\
&=& \sum_{j=0}^I\binom{I}{j}(-it)^{N-j}e^{-i\la t}
\sum_{i=0}^j\binom{j}{i}x^{2k}(-1)^i {L_{m-i}^{\al+i}(|\la|x^2)\over
L_m^\al(0)}
(-{x^2\over2})^{j-i}e^{-{|\la|x^2\over2}}\nonumber\\
&=&\sum_{j=0}^{I}\binom{I}{j}(-it)^{I-j}x^{2j}(-1/2)^{j}
\sum_{i=0}^{j}\binom{j}{i}(-1/2)^{-i}{L_{m-i}^{\al+k}(0)\over
L_{m}^{\al}(0)} \varphi_{(\la,m-i)}^{\al+i}\xt.\nonumber\\
\end{eqnarray}
Now, using the fact  that $|\varphi_{(\la,m-i)}^{\al+i}\xt|\leq1$
and that
$$\Big|\sum_{i=0}^{j}\binom{j}{i}(-1/2)^{-k}{L_{m-i}^{\al+i}(0)\over
L_{m}^{\al}(0)}\Big|\leq  2^{j} $$ it follows
\begin{eqnarray*}
\Big{|}{\partial^I\over\partial\la^I}\varphi^\al_\lam\xt\Big{|}&\leq&
\sum_{j=0}^{I}\binom{n}{j}(|t|)^{I-j}x^{2j}
= (x^2+|t|)^I\\
&\leq& C(x^4+4t^2)^{{I\over2}}=C~\NB\xt^{2I}.
\end{eqnarray*}
An analogous results follows for $\la<0,$ as required to prove inequality (\ref{estimate2}). \\
\textbullet~ The inequality (\ref{estimate3}), follows from (\ref{estimate4}),
use the fact that $|\varphi_{(\la,m-i)}^{\al+i}\xt|\leq C |\varphi_{(\la,m)}^{\al}\xt|.$ As required to finish the proof.
\end{proof}
\begin{lemma}\label{lemma1}
Let $f$ be a $n-$differentiable function over $\mathbb{R}$ and let
$a$  be a real number. Then the function $g(x)=f(ax^2)$ is
$n-$differentiable and satisfies
$$g^{(n)}(x)=\sum_{i=0}^{{n\over2}}\alpha_ia^{n-i}x^{n-2i}f^{(n-i)}(ax^2), \quad \mbox{ if } n \mbox{ is even }.$$
$$g^{(n)}(x)=\sum_{i=0}^{{n-1\over2}}\beta_ia^{n-i}x^{n-2k}f^{(n-i)}(ax^2), \quad \mbox{ if } n \mbox{ is odd }.$$
$(\alpha_i)$  and $(\beta_i)$ are two real sequences.
\end{lemma}
\begin{proof} By induction.\end{proof}
\begin{proposition}\label{p3.3.10}
For all $(x,t)\in \mathbb{K}$ and $(\lambda, m)\in
\widehat{\mathbb{K}}$ the function
$\varphi^{\alpha}_{(\lambda,m)}(x,t)$ satisfies
\begin{equation}
\varphi^{\alpha}_{(\lambda,m)}(x,t)=\sum_{k+\ell\leq I}
\omega_{k,\ell}(\lambda, m, \alpha)~ x^{k}t^{\ell}+R_{\theta}(x,t), \,\,
 0<\theta<1,
\end{equation}
where
\begin{equation}
|R_{\theta}(x,t)|\leq C~\sum_{k+\ell=
I+1}x^{k}|t|^{\ell}\mathcal{N}(\lambda,m)^{\frac{k}{2}+\ell}, \quad
\mbox{if } \, k \quad  \mbox{is even}.
\end{equation}
An analogous  result  follows for $k$ is odd .

Here  $C$ is a constant depending only on $I$ and $ \omega_{k,\ell}(\lambda, m,
\alpha)$ are functions expressed by mean of $\lambda, m$ and $
\alpha.$
\end{proposition}
\begin{proof}
In order  to estimate $R_{\theta}(x,t)$ we may  write
$\varphi^{\alpha}_{(\lambda,m)}(x,t)$  in its expanded  form and
by making use of Taylor series with integral's remainder with
respect to the variables $(x,t)$ together with the  help of  Lemma
\ref{lemma1}, so we get  \small
\begin{eqnarray*}
\varphi^{\alpha}_{(\lambda,m)}(x,t)&=&\sum_{k+\ell\leq
I}\frac{x^{k}}{k!}.\frac{t^{\ell}}{\ell!} \frac{\partial^{k+\ell}}{\partial
x^{k}\partial t^{\ell}}\varphi^{\alpha}_{(\lambda,m)}(0,0)~~+
\sum_{k+\ell= I+1}\frac{x^{k}}{k!}.\frac{t^{\ell}}{\ell!}
\frac{\partial^{k+\ell}}{\partial x^{k}\partial
t^{\ell}}\varphi^{\alpha}_{(\lambda,m)}(\theta x,\theta
t)\\
&=& \sum_{k+\ell\leq
I}x^{k}.t^{\ell}\underbrace{\frac{1}{k!}\frac{1}{\ell!} \lambda^{\ell}
\varphi^{\alpha}_{(\lambda,m)}(0,0)}_{w_{k,\ell}(\lambda,\alpha,m)}
~~~+~~R_{\theta}(x,t).
\end{eqnarray*}
\noindent A direct calculation yields
\begin{eqnarray*}
\frac{\partial^{k+\ell}}{\partial x^{k}\partial
t^{\ell}}\Big[\varphi^{\alpha}_{(\lambda,m)}(x,t)\Big] &=&
(i\la)^\ell e^{i\la t}[\mathcal{L}_m^\al(|\la|x^2)]^{(k)}.
\end{eqnarray*}
To compute  $[\mathcal{L}_m^\al(|\la|x^2)]^{(k)}$ we may use
Lemma \ref{lemma1}, for $k$  being even integer, (An analogous
result  follows for $k$  being odd integer), so we obtain
$$[\mathcal{L}_m^\al(|\la|x^2)]^{(k)}=
\sum_{j=0}^{{k\over2}}\al_j|\la|^
{k-j}x^{k-2j}[\mathcal{L}_m^\al]^{(k-j)}(|\la|x^2).$$
Now using the fact that
$\mathcal{L}_m^\al(z)=e^{z/2}{L_m^\al(z)\over L_{m}^{\al}(0)}$ and
Leibnitz's formula for the $(k-j)$'th derivative of a product
together with recurrence identity for Laguerre polynomials $
\frac{\partial^{n}}{\partial z^{n}} L^{\alpha}_{m}(z) =
(-1)^{n}L^{\alpha +n}_{m-n}(z)$ we obtain
$$[\mathcal{L}_m^\al(z)]^{(k-j)}=
\sum_{s=0}^{k-j}C_{k,s}e^{-z/2}{L_{m-s}^{\al+s}(z)\over L_{m}^{\al}(0)}.$$
Hence, it follows that $|\la
x^2|^{{k\over2}-j}[\mathcal{L}_m^\al]^{(k-j)}(|\la|x^2)$ is
bounded, and thus
\begin{eqnarray*}
|[\mathcal{L}_m^\al(|\la|x^2)]^{(k)}|&\leq&\Bigg|
\sum_{j=0}^{{k\over2}}\al_j|\la|^{k/2}
\underbrace{|\la
 x^2|^{{k\over2}-j}[\mathcal{L}_m^\al]^{(k-j)}(|\la|x^2)}_{bounded}\Bigg|\\
 &\leq& C\sum_{j=0}^{{k\over2}}\al_j|\la|^{k/2}\\
 &\leq&C|\la|^{k/2}.
\end{eqnarray*}
So, it turns out
\begin{eqnarray*}
\Bigg|\frac{\partial^{k+\ell}}{\partial x^{k}\partial
t^{\ell}}\varphi^{\alpha}_{(\lambda,m)}(x,t)\Bigg|
\leq C|\la|^{k/2+\ell}
&\leq&C|\la|^{k/2+\ell}(m+{\al+1\over2})^{k/2+\ell}\\
&=&C{\mathcal{N}}(\la,m)^{k/2+\ell},
\end{eqnarray*}
and hence
\begin{eqnarray*}
|R_{\theta}(x,t)|&\leq &C~\sum_{k+\ell=
I+1}x^{k}~|t|^{\ell}\mathcal{N}(\lambda,m)^{\frac{k}{2}+\ell}.
\end{eqnarray*}
The proof is completed.
\end{proof}
\subsection{Fourier--Laguerre transform on $H^{p}(\mathbb{K})$}

The purpose of this section is to study the Fourier--Laguerre transform
on the Laguerre--Hardy type space $H^p(\mathbb{K}).$ Our results can be seen as an extension of celebrated properties of the Fourier transform
on classical Hardy spaces (see \cite{COW0,COW,GR,TW}).

Obviously, if $g\in \mathcal{S}(\mathbb{K}),$ and
$f=\sum_{k}\beta_k a_k$ is $a (p,2,s)$ atom--decompositional
expression of $f,$ then
$$\big\langle f,\mathcal{F}(g)\big \rangle =\lim_{n \rightarrow \infty}
\Big\langle \sum_{k}^{n}\beta_k a_k, \mathcal{F}(g)\Big\rangle$$
exists. Note that $a_k\in L^{2}_{\alpha}(\mathbb{K}),$ and
$$\Big\langle \sum_{k}^{n}\beta_k a_k, \mathcal{F}(g)\Big\rangle =
\Big\langle \sum_{k}^{n}\beta_k \widehat{a_k}, g\Big\rangle. $$
Then $\mathcal{F}(f)\in \mathcal{S}'(\mathbb{K})$ if we define
$\mathcal{F}(f)=\sum_{k}\beta_k\widehat{a_k}.$ In the following, we shall further point out that the series $\sum_{k}\beta_k\widehat{a_k}$ is absolutely convergent, and $\mathcal{F}(f)$ defined above is actually a continuous function.
In order to prove this results, we need to establish the following Lemma, which is a generalization of \cite[Theorem III.7.20]{GR}.
\begin{lemma}\label{lemma4.4}
Let $0<p\leq 1,$ and let $a(p,2,s)-$atom supported on a $B(e,r)$ centred at the orgin. Set $d=\frac{1}{1/p-1/2}=2p/(2-p).$ Then there exists a constant $C$ independent of $a$ such that
\begin{itemize}
\item[(i)] For every $0\leq I\leq s,$
\begin{equation}\label{4.4.6}
|\mathbf{D}^{I}_{\lambda}\widehat{a}(\lambda,m)|\leq
C~\mathcal{N}(\lambda,m)^{\frac{s-d(I)+1}{2}}\|a\|_{L^{2}_{\alpha}(\mathbb{K})}
^{1-d\big(\frac{s+1}{Q}+\frac{1}{2}\big)}.
\end{equation}
\item[(ii)]  For every $0\leq I\leq s$ and every $1\leq q \leq\infty,$
\begin{equation}\label{4.4.66}
\|~|~\mathbf{D}^{I}_{\lambda}\widehat{a}~|
    ^{2}~\|_{L^{q'}_{\alpha}(\widehat{\mathbb{K}})} \leq C
\|a\|^{2-d(\frac{2d(I)}{Q}+\frac{1}{q})}_{L^{2}_{\alpha}(\mathbb{K})}.
\end{equation}
\end{itemize}
Here $\frac{1}{q}+\frac{1}{q'}=1.$
\end{lemma}
\begin{proof}
$(i)$  Let $a$ be $a(p, 2, s)$--atom with $Supp~a\subset B(e,r).$ By the size condition of $a,$ we know that
$m_{\alpha}B(e,r)\leq \| a\|^{\frac{1}{1/p-1/2}}_{L^{2}_{\alpha}(\mathbb{K})}= \| a\|^{-d}_{L^{2}_{\alpha}(\mathbb{K})},$ and
$ m_\al B(e,r)=\int_{B(e,r)}dm_\al(x,t)=C~r^{Q},$ one has $\|a\|_{L^{2}_{\alpha}(\mathbb{K})}\leq C~ r^{-\frac{Q}{d}};$ that is,
\begin{equation}
r=C ~m_{\alpha}B(e,r)^{Q}\leq C~ \|a\|^{-\frac{d}{Q}}_{L^{2}_{\alpha}(\mathbb{K})}.
\end{equation}
By the cancelation property of atom,
\begin{eqnarray*}
\mathbf{D}_{\lambda}^{I}\widehat{a}(\lambda,m)&=&\int_{\mathbb{K}}
(x,t)^{I}a(x,t)\big[\varphi^{\alpha}_{(\lambda,m)}(x,t)-
P(x,t)\big]dm_{\alpha}(x,t)\\
&=&\int_{B(e,r)}
(x,t)^{I}a(x,t)\big[\varphi^{\alpha}_{(\lambda,m)}(x,t)-
P(x,t)\big]dm_{\alpha}(x,t),
\end{eqnarray*}
where $$ P(x,t)=\sum_{k+\ell\leq s-d(I)}\omega_{k,\ell}~
x^{k}~t^{\ell}~\mathcal{N}(\lambda,m)^{k/2+\ell}.$$
Therefore by Proposition \ref{p3.3.10}, the estimate for the remainder in Taylor's formula yields;
\begin{eqnarray*}
|\mathbf{D}^{I}_{\lambda}\widehat{a}(\lambda,m)|
&\leq&\sum_{k+\ell=s-d(I)+1}\omega_{k,\ell}~\mathcal{N}(\lambda,m)^{k/2+\ell}
\int_{\mathbb{K}}~
x^{k}~t^{\ell}~|a(x,t)|dm_{\alpha}(x,t)\\
&\leq&\sum_{k+\ell=s-d(I)+1}\omega_{k,\ell}~\mathcal{N}(\lambda,m)^{k/2+\ell}
\int_{\mathbb{K}}(x,t)^{k+\ell}~|a(x,t)|dm_{\alpha}(x,t)\\
&\leq&\sum_{k+\ell=s-d(I)+1}\omega_{k,\ell}~\mathcal{N}(\lambda,m)^{k/2+\ell}
\int_{\mathbb{K}}\mathbf{N}(x,t)^{k+\ell}~|a(x,t)|dm_{\alpha}(x,t)\\
&\leq&\sum_{k+\ell=s-d(I)+1}\omega_{k,\ell}~\mathcal{N}(\lambda,m)^{k/2+\ell}
r^{s+1}\int_{B(e,r)}|a(x,t)|dm_{\alpha}(x,t)\\
&\leq&\sum_{k+\ell=s-d(I)+1}\omega_{k,\ell}~\mathcal{N}(\lambda,m)^{k/2+\ell}
r^{s+1}\|a\|_{L^{1}_{\alpha}(\mathbb{K})}\\
&\leq&\sum_{k+\ell=s-d(I)+1}\omega_{k,\ell}~\mathcal{N}(\lambda,m)^{k/2+\ell}
r^{s+1}m_{\alpha}B(e,r)^{\frac{1}{2}}
\|a\|_{L^{2}_{\alpha}(\mathbb{K})}\\
&\leq&\sum_{k+\ell=s-d(I)+1}\omega_{k,\ell}~\mathcal{N}(\lambda,m)^{k/2+\ell}
\|a\|^{1-d\big(\frac{s+1}{Q}+\frac{1}{2}\big)}
_{L^{2}_{\alpha}(\mathbb{K})}\\
&\leq&~C~\mathcal{N}(\lambda,m)^{\frac{s-d(I)+1}{2}}
\|a\|^{1-d\big(\frac{s+1}{Q}+\frac{1}{2}\big)}_{L^{2}_{\alpha}(\mathbb{K})}.
\end{eqnarray*}
$(ii)$ First we deal with the cases $q' = 1$ and $q = \infty,$ and then we use interpolation.\\ \\
$\bullet $ When $q'=1,$ we have $q=\infty,$  by using the
$L^{1}_{\alpha}, ~L^{\infty}_{\alpha}$ boundedness of the
Fourier--Laguerre  transform.
\begin{eqnarray}\label{4.4}
\int_{\widehat{\mathbb{K}}}|\mathbf{D}^{I}_{\lambda}\widehat{a}
(\lambda,m)|^{2}d\gamma_{\alpha}(\lambda,m)&=&
\int_{\mathbb{K}}|a(x,t)|^{2}|(x,t)^{I}|^{2}dm_{\alpha}(x,t)\nonumber\\
&\leq& C r^{d(I)}\|a\|^{2}_{L^{2}_{\alpha}(\mathbb{K})}\nonumber\\
&=&C~m_{\alpha}B(e,r)^{-\frac{2d(I)}{Q}}\|a\|^{2}_{L^{2}_{\alpha}
(\mathbb{K})}\nonumber\\
&\leq &C\|a\|^{2-2d\,\frac{d(I)}{Q}}_{L^{2}_{\alpha}(\mathbb{K})}.
\end{eqnarray}
$\bullet $ If $q'=\infty,$ we have $q=1$ we use Plancherel's theorem to obtain
\begin{eqnarray}\label{4.5}
|\mathbf{D}^{I}_{\lambda}\widehat{a}(\lambda,m)|^{2}&=& |\widehat{(x,t)^{I}a(x,t)}|^{2}\leq C\|(x,t)^{I}a(x,t)\|^{2}_{L^{1}_{\alpha}(\mathbb{K})}\nonumber\\
&\leq&  C r^{2d(I)}\|a\|^{2}_{L^{1}_{\alpha}(\mathbb{K})}\nonumber\\
&\leq&  C r^{2d(I)}m_{\alpha}B(e,r)\|a\|^{2}_{L^{2}_{\alpha}(\mathbb{K})}\nonumber\\
&\leq&  C \|a\|^{2-d(\frac{2d(I)}{Q}+1)}_{L^{2}_{\alpha}(\mathbb{K})}.
\end{eqnarray}
$\bullet $ Now, for $1<q'<\infty,$  and  by (\ref{4.4})--(\ref{4.5}),  we write
\begin{eqnarray}\label{4.6}
\int_{\widehat{\mathbb{K}}}|\mathbf{D}^{I}_{\lambda}
\widehat{a}(\lambda,m)|^{2q'}d\gamma_{\alpha}(\lambda,m)&=&
\int_{\widehat{\mathbb{K}}}|\mathbf{D}^{I}_{\lambda}
\widehat{a}(\lambda,m)|^{2}
|\mathbf{D}^{I}_{\lambda}\widehat{a}(\lambda,m)|^
{2q'-2}d\gamma_{\alpha}(\lambda,m)\nonumber\\
&\leq &C~\|\mathbf{D}^{I}_{\lambda}\widehat{a}\|^{2}_{L^{2}_{\alpha}(\mathbb{K})}
\|~|\mathbf{D}^{I}_{\lambda}\widehat{a}
|^{2(q'-1)}\|_{L^{\infty}_{\alpha}(\mathbb{K})}\nonumber\\
&\leq &C~\|a\|^{q'\big(2-d(\frac{2d(I)}{Q}+\frac{1}{q})\big)}_{L^{2}_{\alpha}
(\mathbb{K})}.
\end{eqnarray}
\end{proof}
\begin{lemma}\label{lemma4.5}
If $a(p,q,s)$--atom then the Fourier--Laguerre transform of $a$ is continuous and satisfies
\begin{equation}\label{fourier a}
|\widehat{a}(\lambda,m)|\leq C~\mathcal{N}(\lambda,m)^{\frac{Q}{2}(\frac{1}{p}-1)}
\end{equation}
\end{lemma}
\begin{proof}
Given $a (p,2,s)$--atom a centered at the origin $e\in\mathbb{K}.$ We use the preceding Lemma \ref{lemma4.4} with
setting $I = 0.$ Equation (\ref{4.4.6}) reduces to the following estimate

\begin{equation}\label{PW-F(a)2}
|\widehat{a}(\lambda,m)|\leq
C~\mathcal{N}(\lambda,m)^{\frac{s+1}{2}}\|a\|_{L^{2}_{\alpha}(\mathbb{K})}
^{1-d\big(\frac{s+1}{Q}+\frac{1}{2}\big)}.
\end{equation}
Also, equation (\ref{4.4.66}) with $q=1,$ yields
\begin{equation}
|\widehat{a}(\lambda,m)|^{2}\leq C
\|a\|^{2-d}_{L^{2}_{\alpha}(\mathbb{K})}.
\end{equation}
Observe that, in the first estimate, the exponent of $\|a\|_{L^{2}_{\alpha}(\mathbb{K})}$ negative, whereas in the second estimate, the exponent of $\|a\|_{L^{2}_{\alpha}(\mathbb{K})}$ is positive, because $d=2p/(2-p) \leq 2.$\\
Now, combining both estimates we obtain a better one.
We use the first one where \linebreak
$\mathcal{N}(\lambda,m)^{\frac{s+1}{2}}\|a\|_{L^{2}_{\alpha}(\mathbb{K})}
^{1-d\big(\frac{s+1}{Q}+\frac{1}{2}\big)}\leq
\|a\|^{1-d/2}_{L^{2}_{\alpha}(\mathbb{K})}$ or, equivalently,
where $\mathcal{N}(\lambda,m)^{\frac{Q}{2}}\leq  \|a\|^{d}_{L^{2}_{\alpha}(\mathbb{K})}.$
In the rest, we use the second estimate.\\
$\bullet$ If $\mathcal{N}(\lambda,m)^{\frac{Q}{2}}\leq  \|a\|^{d}_{L^{2}_{\alpha}(\mathbb{K})},$
we get
\begin{eqnarray*}
|\widehat{a}(\lambda,m)|&\leq & C\mathcal{N}(\lambda,m)^{\frac{s+1}{2}}
\|a\|_{L^{2}_{\alpha}(\mathbb{K})}^{1-d\big(\frac{s+1}{Q}+\frac{1}{2}\big)}\\
&\leq&C \mathcal{N}(\lambda,m)^{\frac{s+1}{2}+
\frac{Q}{2d}\big(1-d(\frac{s+1}{Q}+\frac{1}{2})\big)}\\
&\leq&C \mathcal{N}(\lambda,m)^{\frac{Q}{2}(\frac{1}{p}-1)}.
\end{eqnarray*}
$\bullet$ If $\mathcal{N}(\lambda,m)^{\frac{Q}{2}}> \|a\|^{d}_{L^{2}_{\alpha}(\mathbb{K})},$ we obtain
\begin{eqnarray*}
|\widehat{a}(\lambda,m)|&\leq & C\|a\|^{1-d/2}_{L^{2}_{\alpha}(\mathbb{K})}\\
&\leq & C~\mathcal{N}(\lambda,m)^{\frac{Q}{2d}(1-d/2)}\\
&\leq& C~\mathcal{N}(\lambda,m)^{\frac{Q}{2}(\frac{1}{p}-1)}.
\end{eqnarray*}
\end{proof}\\
From the above Lemma, we can deduce an estimation on the Fourier--Laguerre transforms.
\begin{theorem}\label{theorem6.6}
Let $f \in H^p(\mathbb{K}),~ 0 < p \leq 1.$
Then the Fourier--Laguerre transform $\mathcal{F}(f)$  of $f$ which always makes sense as a tempered distribution, is actually a continuous function satisfying the estimate.
\begin{equation}\label{(4.10)}
|\mathcal{F}(f)(\lambda,m)|\leq C \|f\|_{H^p(\mathbb{K})}
\mathcal{N}(\lambda,m)^{\frac{Q}{2}(\frac{1}{p}-1)}
\end{equation}
with $C$ independent of $f.$
\end{theorem}
\begin{proof}
Let $f\in H^{p}(\mathbb{K}).$ By the atomic decomposition of $H^{p}(\mathbb{K}),$ we can find coefficients $(\beta_k)$ and atoms $(a_k)$ such that $f=\sum_{k}\beta_{k}a_{k}$ (in $H^p(\mathbb{K})$--norm) and $\sum_{k}|\beta_{k}|^p\leq C\|f\|_{H^p(\mathbb{K})}.$
This sum converges in $H^p$--norm, which implies convergence in $\mathcal{S}'(\mathbb{K}).$ So by taking the Fourier--Laguerre transform on $f,$ we have $\mathcal{F}(f)(\lambda,m)=\sum_{k}\beta_{k}\widehat{a_{k}}(\lambda,m),$ converging in $\mathcal{S}'(\mathbb{K}).$  By (\ref{fourier a}) and the fact that $\sum_{k}|\beta_k|<\infty,$
$$\sum_k|\beta_k||\widehat{a_k}(\lambda,m)|\leq C \sum_k|\beta_k|\mathcal{N}(\lambda,m)^{\frac{Q}{2}(\frac{1}{p}-1)}
\leq C \mathcal{N}(\lambda,m)^{\frac{Q}{2}(\frac{1}{p}-1)}
\|f\|_{H^{p}(\mathbb{K})}<\infty.$$
Therefore, the sum $\mathcal{F}(f)(\lambda,m)=\sum_{k}\beta_{k}\widehat{a}_{k}(\lambda,m)$ converges absolutely on $\mathbb{K}.$  Furthermore, on each
compact set $K,~\mathcal{N}(\lambda,m)$ is bounded by a constant $C$ independent of $a,$ so the absolute convergence
above is also uniform on each compact set $K.$ With $\widehat{a}_{k}$ infinitely differentiable (hence continuous) for all $k,$ we conclude $\mathcal{F}(f)$ is continuous on all compact sets $K,$ and hence on
$\mathbb{K}.$
\end{proof}

We now consider consequences of Theorem \ref{theorem6.6}. The first corollary refines the order of 0 at the origin, and the second gives
 weak--type inequality for the Fourier--Laguerre transform. The third
corollary is the Paley--type inequality on Hardy spaces.
\begin{corollary}\label{Cor-Local}
Let $f\in H^{p}(\mathbb{K}),~0<p\leq1.$ Then,
\begin{align}\label{PW-F(f)-origin}
\lim_{\lambda \rightarrow 0} \frac{\mathcal{F}(f)(\lambda,m)}{\mathcal{N}(\lambda,m)^
{\frac{Q}{2}(\frac{1}{p} - 1)}} = 0,\quad \mbox{for all} \quad m\in \mathbb{N}.
\end{align}
\end{corollary}
\begin{proof}
We start by verifing this on an atom $a.$ By \eqref{PW-F(a)2}, we have
$$|\widehat{a}(\lambda,m)|\leq
C~\mathcal{N}(\lambda,m)^{\frac{s+1}{2}}\|a\|_{L^{2}_{\alpha}(\mathbb{K})}
^{1-d\big(\frac{s+1}{Q}+\frac{1}{2}\big)}.$$
Since $s \geq [Q (1/p - 1)],$ this implies $\frac{s + 1}{2}  > \frac{Q}{2}(\frac{1}{p} - 1).$ Therefore, we obtain \eqref{PW-F(f)-origin} for atoms;
$$\lim_{\lambda \rightarrow 0} \frac{\widehat{a}(\lambda,m)}{\mathcal{N}(\lambda,m)^
{\frac{Q}{2}(\frac{1}{p} - 1)}} = 0,\quad \mbox{for all} \quad m\in \mathbb{N}.$$
Now if $f \in H^p(\mathbb{K}),$ we can decompose $f = \sum_{k} \beta_k a_k,$ for $\sum_{k} |\beta_k|^p <\infty$ and $(p, q, s)$--atoms $a_k.$ Thus,
$$\frac{\mathcal{F}(f)(\lambda,m)}{\mathcal{N}(\lambda,m)^{\frac{Q}{2}(\frac{1}{p} - 1)}}\leq \sum_{k = 1}^{\infty} \frac{|\widehat{a_k}(\lambda,m)|}{\mathcal{N}(\lambda,m)^{\frac{Q}{2}(\frac{1}{p} - 1)}} |\beta_k|.$$
By \eqref{(4.10)} and the fact that $\sum_{k} |\beta_k| <\infty,$ we can apply the Dominated Convergence Theorem to the above sum (treated as an integral). Since each term in the sum goes to $0$ as $\lambda \to 0$ for all $m\in\mathbb{N}$ we obtain \eqref{PW-F(f)-origin}.
\end{proof}
\begin{corollary}
Let $f$ be a function in $H^p(\mathbb{K}),$ with $0<p\leq1.$ Then
$$\gamma_{\alpha}\Big(\big\{(\lambda,m)\in\widehat{\mathbb{K}}: ~\mathcal{N}(\lambda,m)^{\frac{Q}{2}(1-\frac{2}{p})}|\mathcal{F}(f)(\lambda,m)|\geq \beta\big\}\Big)\leq C \frac{\|f\|^{p}_{H^{p}(\mathbb{K})}}{\beta^{p}},~~ \beta>0.$$
\end{corollary}
\begin{proof}
Let $f\in H^p(\mathbb{K})$ with $0<p\leq1$  and let $\beta>0.$
We consider the set
$$E^{p}_{\lambda,m}=\{(\lambda,m)\in\widehat{\mathbb{K}} : \mathcal{N}(\lambda,m)^{\frac{Q}{2}(1-\frac{2}{p})}
|\mathcal{F}(f)(\lambda,m)|\geq \beta\big\}.$$
From (\ref{(4.10)}) it follows that
\begin{eqnarray*}
\gamma_{\alpha}\Big(E^{p}_{\lambda,m}\Big)
\leq C \int_{B(e,r_{p})}d\gamma_{\alpha}(\lambda,m)
&\leq&
   C\sum_{m\geq0}L^\alpha_m(0)\int_{{-r_{p}\over 4m+2\alpha+2}}^{r_{p}\over 4m+
   2\alpha+2}|\lambda|^{\alpha+1}d\lambda\\
&=& C~r^{Q/2}_{p} \sum_{m\geq0} {L_m^\alpha(0)\over
(4m +2\alpha+ 2)^{ \alpha+2}}.
\end{eqnarray*}
Using the fact that
$L^{\alpha}_{m}(0)\sim\frac{m^{\alpha}}{\Gamma(\alpha+1)}$  and $r_{p}=\big(\|f\|_{H^{p}(\mathbb{K})}/\beta\big)
^{2p/Q}$ which finishes the proof.
\end{proof}
\begin{corollary}\cite{AM00}(Paley--type inequality)\label{t3.3.11}
Let $0<p \leq 1.$ There exists $C > 0$ such that, for every
$f\in H^p(\mathbb{K}),$ we have
\begin{equation}\label{3.3.13}
\int_{\widehat{\mathbb{K}}}\frac{|\mathcal{F}(f)(\lambda,m)|^{p}}
{|\mathcal{N}(\lambda,m)|^{\frac{Q}{2}(2-p)}}
d\gamma_{\alpha}(\lambda,m)\leq C\|f\|^{p}_{H^{p}(\mathbb{K})},
\end{equation}
where $C$ is independent of $f.$
\end{corollary}
\begin{remarks}
\begin{verse}
\end{verse}
\begin{enumerate}
\item A version of the Hardy--type inequality for  Fourier--Laguerre transform on $\mathbb{K}$ appears when we take $p = 1.$
\item If $\al=n-1, \, (n\in \mathbb{N}-\{0\})$, $Q=2n+2$,
which is nothing but the homogenous dimension of Heisenberg group
$\mathbb{H}^n.$
\end{enumerate}
\end{remarks}
By using the Marcinkiewicz interpolation theorem we obtain an analogue of the Paley--type inequality  was extended (the range of $p$) to $L^{p}_{\alpha}(\mathbb{K}),~1<p\leq2,$
that is a Pitt--type inequality for the Fourier--Laguerre transform.
\begin{theorem}(Pitt--type inequality)
Let $1<p\leq2.$ There exists $C>0$ such that,  for every $f\in L^{p}_{\alpha}(\mathbb{K}),$ we have
$$\int_{\widehat{\mathbb{K}}}\frac{|\mathcal{F}(f)(\lambda,m)|^{p}}
{|\mathcal{N}(\lambda,m)|^{\frac{Q}{2}(2-p)}}
d\gamma_{\alpha}(\lambda,m)\leq C~\|f\|^{p}_{L^{p}(\mathbb{K})}$$
where $C$ is independent of $f.$
\end{theorem}
\begin{proof}Consider the map defined by
$$Tf(\lambda,m):=|\mathcal{N}(\lambda,m)|^{Q/2}|
\mathcal{F}(f)(\lambda,m)|.$$
It follows from Plancherel's theorem for the Fourier--Laguerre transform that the operator $T$ is a strong $(2, 2)$ type, i.e;
$$\|Tf\|_{L^{2}_{\alpha}(\widehat{\mathbb{K}},d\nu_{\alpha})}
\leq C \|f\|_{L^{2}_{\alpha}(\mathbb{K},dm_{\alpha})},\quad \mbox{where}\quad
d\nu_{\alpha}(\lambda,m):=\frac{d\gamma_{\alpha}(\lambda,m)}{|\mathcal{N}
(\lambda,m)|^{Q}}.$$
Now, consider the set  $E^2_{\lambda,m}=\{(\lambda,m)\in\widehat{\mathbb{K}} : T(f)(\lambda,m)>\beta\}.$
Since $\|\mathcal{F}(f)\|_{L^{\infty}_{\alpha}(\widehat{\mathbb{K}})}\leq \|f\|_{L^1_{\alpha}(\mathbb{K})}.$ Therefore, for $(\lambda,m)\in E_{\lambda,m}$ we have $|\lambda|^{Q/2}> \frac{\beta}{(4m+2\alpha+2)^{Q/2}\|f\|_{L^1_{\alpha}(\mathbb{K})}}
:=R^{2/Q}.$ This forces
$$\nu_{\alpha}(E_{\lambda,m})=\int_{E_{\lambda,m}}
\frac{d\gamma_{\alpha}(\lambda,m)}{|\mathcal{N}(\lambda,m)|^{Q}}
\leq C \sum_{m\geq0}L^\alpha_m(0)
\int_{\mathbb{R}\backslash(-R,R)}|\lambda|^{-Q+\alpha+1}d\lambda
= C~R^{-Q/2}.$$
Thus, $\nu_{\alpha}\Big(\{(\lambda,m) : T(f)(\lambda,m)>\beta\}\Big)\leq C \frac{\|f\|_{L^1_{\alpha}(\mathbb{K})}}{\beta},$ which implies $T$ is of weak type $(1, 1).$
Thus, the Marcinkiewicz interpolation theorem implies the theorem.
\end{proof}

\subsection{The Fourier--Laguerre multiplier}

Let $M : \widehat{\mathbb{K}}\longrightarrow \mathbb{C}$ be a bounded function and $T_{M}$ the multiplier operator initially
defined  by
\begin{equation*}
f\longrightarrow T_{M}(f)=\mathcal{F}^{-1}(M\widehat{f}),\qquad
f\in H^{p}(\mathbb{K})\cap \mathcal{S}(\mathbb{K})
\end{equation*}
The aim of this section is to prove an analogue of the famous H\"ormander
multiplier theorem, states as follows.
\begin{theorem}\label{multiplier}
Let $0<p\leq 1$ and $\tau>Q(\frac{1}{p}-\frac{1}{2})$ is even.
Suppose that $M\in \mathcal{C}^{d(I)}(\widehat{\mathbb{K}}),\linebreak 0\leq d(I)\leq\tau,$  is a bounded measurable function satisfying which satisfies the H\"ormander condition :
\begin{equation}\label{Hormcondition}
\int_{\frac{R}{2}\leq \mathcal{N}(\lambda,m)\leq R}
|\mathbf{D}_{\lambda}^{I}M(\lambda,m)|^{2}d\gamma_{\alpha}(\lambda,m)
\leq C~R^{Q-d(I)},\quad \mbox{for all}~R>0.
\end{equation}
Then the operator $T_M $ can be extended a bounded operator on $H^{p}(\mathbb{K}).$
\end{theorem}
\begin{remark}
Note that the H\"ormander condition (\ref{Hormcondition}) is implied by Mihlin's condition :
\begin{equation}\label{Mihlin}
|\mathbf{D}_{\lambda}^{I}M(\lambda,m)|\leq C \mathcal{N}(\lambda,m)^{-\frac{d(I)}{2}},\quad
0\leq d(I)\leq\tau.
\end{equation}
\end{remark}
\begin{proof}
Let a be $a(p,2, s_1)$--atom with $s \leq s_1.$ Let $\varepsilon$ be a constant satisfying the condition of molecules. Then $a$ will be $a (p,2, s,\varepsilon)$ molecule with $\mathfrak{N}(a)\leq C.$ Here C is a
constant independent of $a.$ In particular, if $a$ is supported on a ball centered at the origin and $\tau=Q[1/2+\varepsilon]=Q\beta$ is a positive integer, then by Plancherel's theorem, we know that
\begin{equation*}
\Big\{\|\widehat{a}\|_{L^{2}_{\alpha}(\widehat{\mathbb{K}})}
^{(\frac{1}{2}-\frac{1}{p}
+\frac{\tau}{Q})}\|\mathbf{D}^{d(I)}_{\lambda}\widehat{a}\|_{L^{2}_{\alpha}
(\widehat{\mathbb{K}})}
^{\frac{1}{p}-\frac{1}{2}}\Big\}^{Q/\tau}\leq C,\quad \mbox{where}\quad
0\leq d(I)\leq \tau.
\end{equation*}
Let $a$ be $a (p, 2,\tau-1)$--atom. We want to show that $ \mathcal{F}^{-1}(M\widehat{a})$ is $\linebreak a (p, 2, [Q(\frac{1}{p}-1)], \frac{\tau}{Q}-\frac{1}{2})$--molecule supported on a ball centered at the origin. Moreover,
$$\mathfrak{N}(T_{M}a)\leq C.$$
We first check the size condition for $\mathcal{F}^{-1}(M\widehat{a}).$ By Plancherel's theorem, we just need to show that
\begin{equation}\label{10}
\mathfrak{N}(T_{M}a)=\Big\{\|T_{M}a\|_{L^{2}_{\alpha}(\mathbb{K})}
^{(\frac{1}{2}-\frac{1}{p}
+\frac{\tau}{Q})}\|\mathbf{N}(.,.)^{\tau}T_{M}a\|_{L^{2}_{\alpha}(\mathbb{K})}
^{\frac{1}{p}-\frac{1}{2}}\Big\}^{Q/\tau}\leq C.
\end{equation}
According to Plancherel Theorem,
\begin{eqnarray}
\|T_{M}a\|_{L^{2}_{\alpha}(\mathbb{K})}&=&
\|\widehat{a}~M\|_{L^{2}_{\alpha}(\widehat{\mathbb{K}})}\nonumber\\
&\leq & C \|\widehat{a}\|_{L^{2}_{\alpha}(\widehat{\mathbb{K}})}\nonumber\\
&\leq & C \|a\|_{L^{2}_{\alpha}(\mathbb{K})}.
\end{eqnarray}
Hence (\ref{10}) is followed from
\begin{equation}\label{11}
\|\mathbf{N}(.,.)^{\tau}T_{M}a\|_{L^{2}_{\alpha}(\mathbb{K})}
\leq C \|a\|^{1-\frac{d\tau}{Q}}_{L^{2}_{\alpha}(\mathbb{K})}.
\end{equation}
Because $\NB\xt^\tau=(x^4+4t^2)^{\tau/4}\leq C(|t| + |x|^2)^{\tau/2}$ and
$\frac{\tau}{2}$  is an integer, we have
\begin{eqnarray*}
\|\mathbf{N}(., .)^{\tau}T_{M}a\|_{L^{2}_{\alpha}(\mathbb{K})}&\leq&
C\|(|x|^{2}+|t|)^{\tau/2}T_{M}a\|_{L^{2}_{\alpha}(\mathbb{K})}\\
&\leq&C\|\widehat{(|x|^{2}+|t|)^{\tau/2}T_{M}a}
\|_{L^{2}_{\alpha}(\widehat{\mathbb{K}})}\\
&\leq&C\sum_{d(I)=\tau}\|\mathbf{D}_{\lambda}^{I}(\widehat{T_{M}a})
\|_{L^{2}_{\alpha}(\widehat{\mathbb{K}})}\\
&\leq&C\sum_{d(I)=\tau}\|\mathbf{D}_{\lambda}^{I}(\widehat{a}M)
\|_{L^{2}_{\alpha}(\widehat{\mathbb{K}})}\\
&\leq&C\sum_{d(I')+d(I'')=\tau}\|(\mathbf{D}_{\lambda}^{I'}\widehat{a})
(\mathbf{D}_{\lambda}^{I''}M)
\|_{L^{2}_{\alpha}(\widehat{\mathbb{K}})}.
\end{eqnarray*}
To prove (\ref{11}), we need only to prove
\begin{equation}\label{12}
\|(\mathbf{D}_{\lambda}^{I'}\widehat{a})(\mathbf{D}_{\lambda}^{I''}M)
\|_{L^{2}_{\alpha}(\widehat{\mathbb{K}})}\leq
C \|a\|_{L^{2}_{\alpha}(\mathbb{K})}^{1-\frac{d\tau}{Q}}.
\end{equation}
Let us now consider two cases for proving (\ref{12}). First,
if $d(I') =\tau,$ from (\ref{4.4}), we get
\begin{equation}\label{13}
\|(\mathbf{D}_{\lambda}^{I'}\widehat{a})M\|_{L^{2}_{\alpha}
(\widehat{\mathbb{K}})}
\leq C\|\mathbf{D}_{\lambda}^{I'}\widehat{a}\|_{L^{2}_{\alpha}
(\widehat{\mathbb{K}})}
\leq
C \|a\|_{L^{2}_{\alpha}(\mathbb{K})}^{1-\frac{d\tau}{Q}}.
\end{equation}
Thus, (\ref{12}) holds. Secondly, suppose $0 \leq d(I') <\tau,$ then
\begin{eqnarray}
\|(\mathbf{D}_{\lambda}^{I'}\widehat{a})(\mathbf{D}_{\lambda}^{I''}M) \|^{2}_{L^{2}_{\alpha}(\widehat{\mathbb{K}})}&\leq&
\sum_{m=0}^{\infty}L^{\alpha}_{m}(0)\int_{\mathbb{R}}
|\mathbf{D}_{\lambda}^{I'}\widehat{a}(\lambda,m)|^{2}|
\mathbf{D}_{\lambda}^{I''}M(\lambda,m)|^{2}
|\lambda|^{\alpha+1}d\lambda \nonumber\\
&\leq & \sum_{m=0}^{\infty}L^{\alpha}_{m}(0)\sum_{\ell=-\infty}^{\infty}
\int_{2^{\ell}<|\lambda|\leq 2^{\ell+1}}
|\mathbf{D}_{\lambda}^{I'}\widehat{a}(\lambda,m)|^{2}|
\mathbf{D}_{\lambda}^{I''}M(\lambda,m)|^{2}
|\lambda|^{\alpha+1}d\lambda.\nonumber
\end{eqnarray}
Fix $\ell_{0}$ such that
$2^{\ell}\leq\|a\|^{\frac{2d}{Q}}_{L^{2}_{\alpha}(\mathbb{K})}< 2^{\ell+1}.$ Making use of (\ref{4.4.6}), we get
\begin{eqnarray}\label{15}
&&\sum_{m=0}^{\infty}L^{\alpha}_{m}(0)\sum_{\ell=-\infty}^{\ell_{0}}
\int_{2^{\ell}<|\lambda|\leq 2^{\ell+1}}
|\mathbf{D}_{\lambda}^{I'}\widehat{a}(\lambda,m)|^{2}|
\mathbf{D}_{\lambda}^{I''}M(\lambda,m)|^{2}
|\lambda|^{\alpha+1}d\lambda\nonumber\\
&\leq& C \sum_{m=0}^{\infty}L^{\alpha}_{m}(0)\sum_{\ell=-\infty}^{\ell_{0}}
\int_{2^{\ell}<|\lambda|\leq 2^{\ell+1}}
|\mathbf{D}_{\lambda}^{I'}\widehat{a}(\lambda,m)|^{2}
|\mathcal{N}(\lambda,m)|^{-d(I'')}
|\lambda|^{\alpha+1}d\lambda\nonumber\\
&\leq& C\|a\|^{2-d(\frac{2\tau}{Q}+1)}_{L^{2}_{\alpha}(\mathbb{K})}
\sum_{m=0}^{\infty}L^{\alpha}_{m}(0)\sum_{\ell=-\infty}^{\ell_{0}}
\int_{2^{\ell}<|\lambda|\leq 2^{\ell+1}}
|\mathcal{N}(\lambda,m)|^{\tau-d(I')-d(I'')}
|\lambda|^{\alpha+1}d\lambda\nonumber\\
&\leq& C\|a\|^{2-d(\frac{2\tau}{Q}+1)}_{L^{2}_{\alpha}(\mathbb{K})}
\sum_{m=0}^{\infty}L^{\alpha}_{m}(0)\sum_{\ell=-\infty}^{\ell_{0}}
\int_{2^{\ell}<|\lambda|\leq 2^{\ell+1}}
|\lambda|^{\alpha+1}d\lambda\nonumber\\
&\leq& C\|a\|^{2-d(\frac{2\tau}{Q}+1)}_{L^{2}_{\alpha}(\mathbb{K})}
\sum_{m=0}^{\infty}L^{\alpha}_{m}(0)\sum_{\ell=-\infty}^{\ell_{0}}
\int_{2^{\ell}<|\lambda|\leq 2^{\ell+1}}
\frac{2^{(\ell+1)(\alpha+1)}}{(4m+2\alpha+2)^{\alpha+1}}d\lambda
\nonumber\\
&\leq& C\|a\|^{2-d(\frac{2\tau}{Q}+1)}_{L^{2}_{\alpha}(\mathbb{K})}
\sum_{m=0}^{\infty}L^{\alpha}_{m}(0)\sum_{\ell=-\infty}^{\ell_{0}}
\frac{2^{(\ell+1)(\alpha+2)}}{(4m+2\alpha+2)^{\alpha+2}}
\quad\Big(\mbox{use}~L^{\alpha}_{m}(0)\sim\small{\frac{m^{\alpha}}
{\Gamma(\alpha+1)}}\Big)\nonumber\\
&\leq& C\|a\|^{2-d(\frac{2\tau}{Q}+1)}_{L^{2}_{\alpha}(\mathbb{K})}
2^{\ell_{0}(\alpha+2)}
\leq C\|a\|^{2-\frac{2d\tau}{Q}}_{L^{2}_{\alpha}(\mathbb{K})}.
\end{eqnarray}
By (\ref{4.4}), we have
\begin{eqnarray}\label{16}
&&\sum_{m=0}^{\infty}L^{\alpha}_{m}(0)\sum_{\ell=\ell_{0}+1}^{\infty}
\int_{2^{\ell}<|\lambda|\leq 2^{\ell+1}}
|\mathbf{D}_{\lambda}^{I'}\widehat{a}(\lambda,m)|^{2}|\mathbf{D}_{\lambda}
^{I''}M(\lambda,m)|^{2}
|\lambda|^{\alpha+1}d\lambda\nonumber\\
&\leq& C \sum_{m=0}^{\infty}L^{\alpha}_{m}(0)\sum_{\ell=\ell_{0}+1}^{\infty}
\int_{2^{\ell}<|\lambda|\leq 2^{\ell+1}}
|\mathbf{D}_{\lambda}^{I'}\widehat{a}(\lambda,m)|^{2}
|\mathcal{N}(\lambda,m)|^{-d(I'')}
|\lambda|^{\alpha+1}d\lambda\nonumber\\
&\leq& C \sum_{\ell=\ell_{0}+1}^{\infty}2^{-\ell d(I'')}\sum_{m=0}^{\infty}L^{\alpha}_{m}(0)
\int_{2^{\ell}<|\lambda|\leq 2^{\ell+1}}
|\mathbf{D}_{\lambda}^{I'}\widehat{a}(\lambda,m)|^{2}
|\lambda|^{\alpha+1}d\lambda\nonumber\\
&\leq& C~2^{-\ell_{0} d(I'')}\|\mathbf{D}_{\lambda}^{I'}\widehat{a}\|^{2}
_{L^{2}_{\alpha}(\widehat{\mathbb{K}})}\leq
C2^{-\ell_{0}d(I'')}\|a\|^{2-2d\,\frac{d(I')}{Q}}
_{L^{2}_{\alpha}(\mathbb{K})}\nonumber\\
&\leq& C\|a\|^{2-\frac{2d\tau}{Q}}
_{L^{2}_{\alpha}(\mathbb{K})}.
\end{eqnarray}
This completes the proof of (\ref{12}) is followed from (\ref{13})--(\ref{16}).
Next, we want to show that $T_{M}a$ satisfies the vanishing moment condition i.e.,
\begin{equation}\label{4.25}
\int_{\mathbb{K}}T_{M}a(x,t)(x,t)^{I}dV(x,t)=0,~~0\leq d(I)\leq [Q(\frac{1}{p}-1)].
\end{equation}
From (\ref{4.4.6}), if $d(I')+d(I'')\leq[Q(\frac{1}{p}-1],$ then
\begin{eqnarray*}
\|(\mathbf{D}_{\lambda}^{I'}\widehat{a})(\mathbf{D}_{\lambda}^{I''}M) \|^{2}_{L^{1}_{\alpha}(\widehat{\mathbb{K}})}&\leq&
\|(\mathbf{D}_{\lambda}^{I'}\widehat{a}) \|^{2}_{L^{2}_{\alpha}(\widehat{\mathbb{K}})}
\|(\mathbf{D}_{\lambda}^{I''}\widehat{a}) \|^{2}_{L^{2}_{\alpha}(\widehat{\mathbb{K}})}\\
&\leq& C \mathcal{N}(\lambda,m)^{\tau-1-d(I')+1}
\mathcal{N}(\lambda,m)^{-d(I'')}
\|a\|^{2-\frac{2d\tau}{Q}}_{L^{2}_{\alpha}(\mathbb{K})}\\
&\leq& C \mathcal{N}(\lambda,m)^{\tau-d(I')-d(I'')}
\|a\|^{2-\frac{2d\tau}{Q}}_{L^{2}_{\alpha}(\mathbb{K})}\\
&\leq& C |\lambda|^{\tau-d(I')-d(I'')}\rightarrow 0~~\mbox{as}~~\lambda\rightarrow 0,~~\mbox{for all} ~ m\in\mathbb{N}.
\end{eqnarray*}
Hence $$\mathbf{D}_{\lambda}^{I}(\widehat{a}(\lambda,m)M(\lambda,m))\rightarrow 0~~~~ \mbox{as}~~\lambda\rightarrow 0,~~0\leq d(I)\leq[Q(\frac{1}{p}-1)],
~~\mbox{for all} ~ m\in\mathbb{N},$$
in the sense of weak convergence, which complete the proof of (\ref{4.25}).
Finally, for general $f \in H^p(\mathbb{K}),$ we know that
$$f=\sum_{k}\beta_{k}a_{k},$$
where $a_{k}$'s are $(p,2,\tau-1)$--atoms and $\sum_{k}|\beta_{k}|^{p}\leq C\|f\|^{p}_{H^{p}(\mathbb{K})}.$ Then by the
above result, we have
$$T_{M}(f)=\mathcal{F}^{-1}(M\widehat{f})=
\sum_{k}\beta_{k}\mathcal{F}^{-1}(M\widehat{a}_{k}).$$
It follows that
\begin{eqnarray*}
\|\mathcal{F}^{-1}(M\widehat{f})\|_{H^p(\mathbb{K})}&\leq&
\sum_{k}|\beta_{k}|~.~\|\mathcal{F}^{-1}(M\widehat{a}_{k})\|_{H^p(\mathbb{K})}
\leq C\sum_{k}|\beta_{k}| ~.~\mathfrak{N}\big(\mathcal{F}^{-1}(M\widehat{a}_{k})\big)\\
&\leq& C\sum_{k}|\beta_{k}| \leq C \Bigg(\sum_{k}|\beta_{k}|^{p}\Bigg)^{1/p}\leq C\|f\|_{H^p(\mathbb{K})}.
\end{eqnarray*}
The proof of this theorem is therefore complete.
\end{proof}

\subsection{Applications}

We are now in a position to state our application, which has been inspired by \cite[see $\S$.4]{DeM}.

The positive symmetric in $L^{2}_{\alpha}$ operator $\mathfrak{L}_{\alpha}$ being hypoelliptic admits a selfadjoint extension. Denote it by $\overline{\mathfrak{L}_{\alpha}}.$ Let $E_{\lambda}$ be the spectral resolution for  $\overline{\mathfrak{L}_{\alpha}},$ (see, \cite{ST1}) i.e.,
$$ \overline{\mathfrak{L}_{\alpha}}f=\int_{0}^{\infty}\lambda dE_{\lambda}(f),\quad f\in Dom(\overline{\mathfrak{L}_{\alpha}}).$$
Then, according to the Littlewood--Paley--Stein theory (see, \cite{St}), if
\begin{equation}
f(\lambda)=\int_{0}^{\infty}e^{-\lambda s}\phi(s)ds
\end{equation}
for some $\phi\in L^{\infty}(0,\infty),$ the operator
$f(\mathfrak{L}_{\alpha})= \int_{0}^{\infty}f(\lambda)dE(\lambda)$ is bounded on $L^{p}_{\alpha}(\mathbb{K}),~p<\infty.$

As a simple corollary of Theorem \ref{multiplier}, the following corollary is convenient for application.
\begin{corollary}
Suppose $f$ is a function of class $\mathcal{C}^{\tau}$ such taht $|f^{(j)}(r)|\leq C~r^{-j}$ for $0\leq j\leq\tau.$ Set
$$M(\lambda,m)=f(\mathcal{N}(\lambda,m)). $$
Then the operator $T_{M}$ can be extended a bounded operator on $H^{p}(\mathbb{K}).$
\end{corollary}
\begin{example}
The fractional operators $\mathfrak{L}^{i \,s},~(I+\mathfrak{L})^{i \,s} ,~s\in\mathbb{R},$ are bounded on $H^p(\mathbb{K}),~0<p\leq1,$  where
$\widehat{\mathfrak{L}^{i \,s}}(f)(\lambda,m)=\mathcal{N}(\lambda,m)^{i \,s}\widehat{f}(\lambda,m)$ and
$\widehat{(I+\mathfrak{L})^{i \,s}(f)}(\lambda,m)=(1+\mathcal{N}(\lambda,m))^{i \,s}\widehat{f}(\lambda,m).$
\end{example}

\end{document}